\documentclass[11pt,a4paper]{article}

\usepackage[T1]{fontenc}
\usepackage{fix-cm}
\usepackage{amsmath, amssymb, amsthm, amsfonts,mathrsfs,stackrel, xurl}
\usepackage{float}
\usepackage{geometry}
\usepackage[hidelinks]{hyperref}
\usepackage{tikz-cd}
\usepackage{enumitem}
\usepackage[normalem]{ulem}
\usepackage{titling}

\let\mathscr\mathscr

\newtheorem{theorem}{Theorem}[section]
\newtheorem{proposition}[theorem]{Proposition}

\newtheorem{remark}[theorem]{Remark}
\newtheorem{definition}[theorem]{Definition}

\newcommand{\C}{\mathbf{C}}

\newcommand{\N}{\mathbf{N}}
\renewcommand{\P}{\mathbb{P}}

\newcommand{\R}{\mathbf{R}}
\newcommand{\Z}{\mathbf{Z}}

\newcommand{\cO}{\mathcal O}

\newcommand{\Fone}{\mathbf{F}_{1}}
\newcommand{\Ftwo}{\mathbf{F}_{1^2}}

\def\spz{\Spec \Z}
\def\spzb{\overline{\Spec \Z}}
\def\spzf{(\Spec \Z)_{\mathbf{F}_{1}}}

\newcommand{\nat}{\widehat{\N^\times}}
\newcommand{\nato}{\widehat{\N^\times_0}}
\newcommand{\natg}{[\P^1_{\Fone}/\langle\alpha\rangle]}
\newcommand{\minf}{\mathscr C_\infty}
\newcommand{\natgi}{[\P^1_{\Ftwo}/\langle\alpha\rangle]}

\newcommand{\natf}{\mathscr S_0}

\newcommand{\frob}{{\operatorname{Fr}}}

\newcommand{\norm}{{\|\cdot\|}}

\newcommand{\Pt}{{\rm Pts}}

\newcommand{\Hom}{{\rm{Hom}}}

\newcommand{\Spec}{\operatorname{Spec}}

\newcommand{\Pic}{\operatorname{Pic}} \newcommand{\Jac}{\operatorname{Jac}}

\newcommand{\Gal}{\mathrm{Gal}}

\def\ie{{\it i.e.\/}\ }
\def\eg{{\it e.g.\/}\ }

\def\cf{{\it cf.\/}\ }

\title{\Large
The Absolute Twistor Line and the Geometry of
$\overline{\Spec\Z}$}

\author{Alain Connes and Caterina Consani}
\date{}

\begin{document}

\maketitle

\begin{abstract}

We construct the absolute algebraic geometry of the compactification \(\overline{\Spec \Z}\) by amalgamating the affine absolute curve \(\spzf\) with an archimedean component defined over the signed extension \(\Ftwo\) of \(\Fone\). By adjoining a formal imaginary unit to the absolute projective line, we obtain an equivariant topos endowed with a canonical geometric inversion symmetry, which induces the twistor real structure on its complex points.
This archimedean geometry is  incorporated into a global absolute curve defined as an internal object of the odd arithmetic topos, dual to the multiplicative monoid of odd positive integers, and governed by the intrinsic Hopf structure of spherical \(\Ftwo\)-algebras.
The restriction of the absolute Frobenius action to odd integers is forced arithmetically by the extension of scalars to \(\Ftwo\). On complex points, the resulting dynamics simultaneously generates the Adams operations and complex conjugation on real Hodge structures.
At the categorical level, the odd arithmetic topos originates in the pericyclic category, whose $\lambda$-operations provide a conceptual interpretation of the local factors of geometric L-functions.

\end{abstract}

\section{Introduction}
In this paper we construct the local geometry of the Zariski completion
$\spzb$ at the archimedean place as an absolute twistor line. We show that
this geometry is governed by the intrinsic Hopf structure of spherical
$\Ftwo$-algebras. 
We then introduce a global absolute curve over $\Ftwo$ by amalgamating the
affine absolute structure of $\spz$ with this new archimedean component.
The resulting curve is endowed with the canonical Frobenius action
$
x\longmapsto x^n$, $n\in\mathbb N_{\mathrm{odd}}^{\times}
$, 
whose restriction to the multiplicative monoid of odd integers is intrinsic
to the $\Ftwo$-structure.
\newline
The construction of the absolute avatar of $\spz$, \ie the affine curve
$\spzf$ defined in~\cite{CC1}, was guided by the description of the Picard monoid
$\Pic(\spz)$ in terms of rank-one torsion-free abelian groups. This guiding
principle is embodied in the following commutative diagram:
\begin{equation}\label{top2ad}
\begin{tikzcd}[row sep=3em, column sep=3em]
&\spz \arrow[dl,"\Theta"']\arrow[dr,"\theta"] & \\
\nato \arrow[rr, "\kappa"]& & \Pic(\spz)
\end{tikzcd}
\end{equation}
The map $\kappa$
is the extension of the canonical bijection
$
\Pt(\nat)\stackrel{\sim}{\longrightarrow}\Pic(\spz)$, 
by sending the base point $\{0\}$ of the arithmetic topos $\nato$ to the generic point\footnote{$\Z$ is the image via $\theta$ of the generic point of $\spz$, and it reflects the density of finite ideles in finite rational adeles.}
$\Z$ of $\Pic(\spz)$. This extension provides the starting point for the
archimedean construction considered below.

Moving from the affine curve $\spzf$ to its completion $\spzb$
requires an extension of the diagram~\eqref{top2ad}, \emph{both} at the
level of the underlying topos and of its structure sheaf.

On the geometric side, the role of $\Pic(\spz)$ is naturally assumed by the
Jacobian $\Jac(\spzb)$. Under the extended map
\[
\tilde{\theta}:\spzb\longrightarrow\Jac(\spzb),
\]
the archimedean point
$
\infty\in\spzb
$
is mapped to the point
\[
\tilde\theta(\infty):=(\Z,\norm=0)\in\Jac(\spzb).
\]
The rank-one abelian group $\Z$ is here endowed with the trivial
archimedean seminorm and carries \emph{no distinguished choice of
ordering}. It is precisely this absence of a preferred ordering that determines the
archimedean construction developed in this paper.

At the adelic level, passing to the completion introduces exactly one additional
parameter. More precisely,
\begin{equation}\label{dec}
\Jac(\spzb)=\Pic(\spz)\times\{\norm=0,\;\norm\neq0\},
\end{equation}
where the second factor records whether the archimedean seminorm is trivial or
nontrivial. The genuine geometric content arises only after lifting this
adelic description to the level of the underlying topos and its structure
sheaf.

To lift the product decomposition~\eqref{dec} to the level of the underlying
topos, the essential distinction between the points of $\nato$ and those of
$\Pic(\spz)$ must  be taken into account. Whereas a point of $\Pic(\spz)$ is represented by a rank-one
torsion-free abelian group $H$, a point of the arithmetic topos $\nato$
consists of an \emph{ordered} group
$
(H,H_+)$. 
The choice of an ordering is an essential constituent of the local geometry,
since it determines the stalk 
\[
\Fone[T^{H_+}]
\]
of the structure sheaf of the arithmetic site $(\nato,\Fone[T])$ (see \cite{CC1}).
At the archimedean place, the absence of a distinguished ordering is resolved
by passing to an equivariant geometry, which constitutes the main new
construction of this paper.\vspace{.05in}

 The archimedean point $\tilde\theta(\infty)\in\Jac(\spzb)$ carries no distinguished choice of ordering, thus passing to the topos level
 requires resolving this unordered group into two possible
orderings:
$
\Z_{\ge 0}$, and $\Z_{\le 0}$. Correspondingly, the local structure of the sheaf is described by two affine charts
with coordinate spherical algebras
\[
\Fone[T^{\Z_{\ge0}}]
\qquad\text{and}\qquad
\Fone[T^{\Z_{\le0}}].
\]
These charts are glued together along their common intersection, corresponding to the
canonical embeddings of the two ordered monoids into the underlying 
group $\Z$. Thus the absence of a distinguished ordering at the archimedean
place naturally gives rise to the absolute projective line
$\P^1_{\Fone}$ constructed in~\cite{CC10}.

The  canonical isomorphism of the two orderings
determines the geometric involution
\begin{equation}\label{alpha}
\alpha:\P^1_{\Fone}\longrightarrow\P^1_{\Fone},
\end{equation}
which exchanges the two affine charts by reversing the ordering. 
This leads naturally to the definition of the \emph{archimedean equivariant topos} 
\[
[\P^1_{\Fone}/\langle\alpha\rangle].
\]

The underlying topological space of $\natg$ is the
Sierpi\'nski space
$
\{\norm=0,\;\norm\neq0\}$,
consisting of an orbifold generic point
$
\eta_\infty
$
and a closed point
$
\gamma$. 

The product decomposition~\eqref{dec} of the Jacobian has a natural
counterpart at the level of the underlying topos. This is the \emph{extended
arithmetic topos}:
\begin{equation*}
\natf:=\nato\times\natg,
\end{equation*}
where the factor $\natg$ defines the local geometry of the
archimedean place. 

The commutative diagram~\eqref{top2ad} therefore extends
to:
\begin{equation*}
\begin{tikzcd}[row sep=3em, column sep=6em]
\spz \arrow[r, hook] \arrow[d, "\Theta"'] \arrow[dd, "\theta"', bend right=40] & \spzb \arrow[d, "\tilde{\Theta}"] \arrow[dd, "\tilde{\theta}", bend left=73] \\
\nato \arrow[r, hook, shift left=0.7ex] \arrow[r, hook, shift right=0.7ex] \arrow[d, "\kappa"'] & \natf = \nato \times \natg \arrow[d, "\tilde{\kappa}=\kappa\times \pi"] \\
\Pic(\spz) \arrow[r, hook, shift left=0.7ex] \arrow[r, hook, shift right=0.7ex] & \Jac(\spzb) = \Pic(\spz) \times \{0, \norm\neq 0\}
\end{tikzcd}
\end{equation*}
where the extended map
$
\widetilde{\kappa}=\kappa\times\pi
$
is the product of $\kappa$ with the canonical projection
\[
\pi:\natg\longrightarrow\{\norm=0,\norm\neq0\}.
\]

The two parallel horizontal arrows in the middle row of the diagram are the canonical
sections of the product topos. The map $\tilde{\Theta}$ extends the original
 $\Theta$ by assigning to every  prime
$
p\in\spz$ 
the point:
\[
\tilde\Theta(p)=(\Theta(p),\eta_\infty)\in\natf,
\]
where the second component records the nontrivial archimedean seminorm.\vspace{.05in}

The genuinely new feature is the image of the archimedean point. The map
$\tilde{\Theta}$ sends
$
\infty\in\spzb
$
to the point
\[
\tilde\Theta(\infty)=(\{0\},\gamma)\in\natf.
\]
Here, the first coordinate is the generic point $\{0\}$ of $\nato$, expressing that the archimedean place is not localized at any finite
prime. The the second coordinate is the closed point $\gamma$ of  $\natg$, that records the vanishing of the archimedean
seminorm.\vspace{.05in}

The geometric construction described above captures the topological and
combinatorial structure of the archimedean place, and  defines the  projective line expected from the
archimedean geometry of a genus zero curve (\cf \cite{CC3,CC5} and \cite{S}).\vspace{.05in} 

The next step  implements the symmetry $\alpha$ in \eqref{alpha} at the structure sheaf level of the archimedean equivariant topos. In place of the naive sheaf symmetry $T \mapsto \frac 1T$, the conceptual choice, guided by the Weyl symmetry, selects the map:
\[
T \mapsto -\frac 1T.
\]
This requires an extension of scalars from $\Fone$ to the quadratic extension\footnote{$\Fone$ is the spherical algebra of the pointed monoid $\{0,1\}$}: 
\begin{equation}\label{F2}
\Ftwo:=\Fone[\{0,1,\epsilon:~\epsilon^2 = 1\}].
\end{equation} 
The extension from $\Fone$ to $\Ftwo$ is  what allows the absolute
 theory to retain the minimal additive datum---the canonical sign $\epsilon$---while
remaining ``below'' the full additive structure of classical rings.
The use of $\Ftwo$  has a few precedents in our
work, where it plays a central role in the development of absolute
geometry (see \cite{CC9,CC7}). 

In this construction, we see $\Ftwo$ as the natural absolute base,
or spherical algebra, equipped with a distinguished central sign element
$\epsilon$.

Every commutative ring $R$ carries a canonical $\Ftwo$-algebra structure,
determined by the assignment
$
\epsilon\longmapsto -1_R$, which gives a canonical morphism of $\Fone$-algebras 
\[
s:\Ftwo \to HR, \qquad s(\epsilon)=-1_R\in HR(1_+) 
\]
that retains precisely the canonical sign and the additive involution
$
x\longmapsto -x$.\vspace{.05in}

Working over $\Ftwo$, we may adjoin a formal imaginary unit---an
\emph{imaginary generator} $J$---subject to the signed relation:
\[
J^2=\epsilon.
\]

We recall from \cite{CC10} the definition of the underlying topological space of $\P^1_{\Fone}$. It consists of two closed points, denoted $+$ and $-$, and a single generic point $\eta$. The non-trivial open sets are $\Omega_+ = \{+, \eta\}$, $\Omega_- = \{-, \eta\}$, and their intersection is denoted $\Omega = \{\eta\}$.
\vspace{.05in}

The structure sheaf of the archimedean   equivariant topos   is the sheaf $\cO$ of spherical algebras over $\Ftwo$ whose sections on the open sets $\Omega_\pm$ are defined by adjoining local generators $J_+$ and $J_-$  as follows:
\begin{equation*}
\begin{aligned}
\mathcal{O}(\Omega_+) &:= \Ftwo[T, J_+] \big/ (J_+^2 = \epsilon) \\
\mathcal{O}(\Omega_-) &:= \Ftwo[S, J_-] \big/ (J_-^2 = \epsilon)
\end{aligned}
\end{equation*}
On the common intersection $\Omega = \{\eta\}$, the transition maps reflect the inversion of the ordering and the contravariant twisting of the imaginary generator:
\begin{equation}\label{eq:transition_maps}
S = T^{-1} \qquad \text{and} \qquad J_- = \epsilon J_+.
\end{equation}
The   automorphism $\alpha$  in \eqref{alpha}, topologically swaps the closed points: $\alpha(+) = -$, and $\alpha(-) = +$. The induced  action of $\alpha$ on the  sheaf $\cO$  is given (see Definition \ref{def:alpha_symmetry})  by the $\Ftwo$-algebra isomorphisms $\alpha^*_{\pm}: \mathcal{O}(\Omega_\mp) \to \mathcal{O}(\Omega_\pm)$:
\begin{alignat*}{2}
\alpha^*_+(S) &:= \epsilon T, &\qquad \alpha^*_+(J_-) &:= J_+
\\
\alpha^*_-(T) &:= \epsilon S, &\qquad \alpha^*_-(J_+) &:= J_-
\end{alignat*}
This definition is  compatible with the contravariant transition map on $\Omega$, since applying $\alpha^*$ to the equation $J_- = \epsilon J_+$ yields the equivalent form $J_+ = \epsilon J_-$,  since $\epsilon^2 = 1$.\vspace{.05in}

Next definition introduces the main geometric object of the present
paper

\begin{definition}\label{local}
We define
\[
\minf:= (\natgi, \cO)
\]
as an equivariant topos over $\Ftwo$ with a selected structure sheaf.  
\end{definition}

The next result identifies the complex points of this topos with those of the twistor projective line $\widetilde{\P}^1$ (see \eg \cite{Si}, \cite{S}), including its canonical real structure.

\begin{theorem}[Theorem~\ref{thm:twistor_space_scheme}] The following statements hold
\begin{enumerate}
    \item The set $\minf(\C)$ is in canonical bijection  with the classical projective line $\P^1(\C)$.
    \item The canonical action of the complex conjugation  on  $\minf(\C)$ induces  the fixed-point-free involution:
\begin{equation*}
z \longmapsto -\frac{1}{\bar{z}}.
\end{equation*}
\end{enumerate}
\end{theorem}

Thus the twistor  line $\widetilde{\P}_\R^1$ is rooted over $\Ftwo$ and arises canonically from the action of the standard complex conjugation on the geometrically folded absolute line. \vspace{.03in}

 Definition~\ref{def:absolute_curve0} introduces the global absolute curve:
$$\bar{\mathscr{C}} = (\overline{\Spec \Z})_{\Ftwo}$$ by amalgamating the finite absolute curve $(\spz)_{\Fone}$ and the archimedean component $\minf$. This gluing is made possible by the orbifold symmetry at infinity, which collapses the generic archimedean stalk to $\Ftwo$,  matching the generic stalk of the finite primes. 

To ensure the global rigidity of this space, we define the absolute curve $\bar{\mathscr C}$ as a dynamical geometric object \emph{internal} to the odd arithmetic topos $\widehat{\N^\times_{\text{odd}}}$.\newline
The $\Ftwo$-enrichment—specifically the twistor relation $J^2 = -1$—arithmetically forces the absolute Frobenius endomorphisms $\frob_n(x) = x^n$ to be restricted to odd integers $n$.\newline By evaluating global sections as morphisms from the terminal object in this topos, Proposition~\ref{prop:global_sections_frobenius} shows that global sections are tautologically the simultaneous fixed points of the odd Frobenius action. This rigidly reduces the global sections to the base algebra $\Ftwo$, providing a purely absolute-algebraic realization of the  compactness of $\spzb$.\newline
Proposition~\ref{prop:twistor_frobenius} states that the odd arithmetic monoid induces a highly non-trivial \sloppy branched dynamical system on the complex twistor line $\minf(\C)$. The action  preserves the quaternionic real structure, acting as $z \mapsto z^n$ for $n \equiv 1 \pmod 4$, and as $z \mapsto -1/z^n$ for $n \equiv 3 \pmod 4$. 
This result is a geometric instance of the fact that the absolute geometry of $\overline{\Spec \Z}$ intrinsically encodes the ramification of the prime $2$ in the arithmetic of the Gaussian integers, seamlessly bridging absolute algebra, topos theory, and complex dynamics.\vspace{.05in}

The dynamical system induced by the odd arithmetic topos on the archimedean twistor line suggests a deep connection to real Hodge theory. In Simpson's twistor construction \cite{Si}, a  real Hodge structure is encoded by a $\C^\times$-equivariant vector bundle on $\mathbb{P}^1(\C)$, where the fibers at $0$ and $\infty$ encode the Hodge filtration and the conjugate filtration, respectively. 
Evaluating the pullback of such a twistor bundle along the odd Frobenius action $\frob_n$ yields a striking cohomological interpretation. For $n \equiv 1 \pmod 4$, the map $z \mapsto z^n$ fixes the poles, and the pullback geometrically realizes the Adams operation $\psi^n$, scaling the Hodge weights. For $n \equiv 3 \pmod 4$, the map $z \mapsto -1/z^n$ applies the Adams operation but simultaneously swaps the poles $0 \leftrightarrow \infty$. In the language of twistor bundles, this corresponds exactly to passing to the conjugate filtration. Thus, the absolute geometry of the odd arithmetic topos dynamically generates both the $\lambda$-ring structure (Adams operations) and the complex conjugation of real Hodge structures.\vspace{.03in}

Finally, it is worth noting that the emergence of the odd arithmetic topos in this
construction has a natural categorical antecedent: it is already implicit in
the pericyclic category $\Pi$ of \cite{CC4}. Through the associated $\lambda$-operations, this structure also underlies the cyclic-cohomological interpretation of the local factors of geometric L-functions (see \cite{CC8} and \cite{H}).

\section{Spherical Algebras and Intrinsic Hopf Structures over
\texorpdfstring{$\Ftwo$}{F1-squared}}\label{sect2}

In classical algebraic geometry over a field or, more generally, over a ring,
a space carries no natural group or monoid structure unless such a structure
is explicitly specified, as in the case of algebraic groups. Absolute
geometry over $\Fone$ and $\Ftwo$, by contrast, exhibits a remarkable
rigidity: spaces arising from spherical monoid algebras naturally inherit
canonical Hopf structures. It is this intrinsic feature that underlies the
dynamical systems and equivariant bundles on the absolute curve.

\subsection{The Lydakis-Day Product and Spherical Algebras}

The foundational objects of absolute geometry are commutative monoids $M$ equipped with an absorbing zero element $0$ and an identity $1$. We refer to the associated absolute algebra $\Fone[M]$ (and its signed enrichment $\Ftwo[M]$) as a \emph{spherical algebra}. 

In the category of $\Fone$-algebras, the correct tensor product is the Lydakis-Day convolution product. For the spherical algebras associated to pointed monoids $M$ and $N$, this categorical tensor product reduces  to the smash product of the underlying monoids:
\begin{equation*}
\Fone[M] \otimes_{\Fone} \Fone[N] \cong \Fone[M \wedge N],
\end{equation*}
where the smash product $M \wedge N := (M \times N) / (M \times \{0\} \cup \{0\} \times N)$ canonically identifies the absorbing zeros. 

When passing to the enriched base $\Ftwo:=\Fone[\{0,1,\epsilon:~\epsilon^2 = 1\}$, the tensor product $\otimes_{\Ftwo}$ is similarly governed by the smash product, but with the additional $\Ftwo$-\emph{linear relations} which allow the canonical sign $\epsilon = -1$ to commute with the tensor symbol:
\begin{equation*}
\epsilon \otimes 1 = 1 \otimes \epsilon = \epsilon.
\end{equation*}

\subsection{The Coproduct and the Canonical Monoid of Points}\label{sect2.2}

For a space defined by spherical algebras over $\Fone$, the set of its rational points over any target $\Fone$-algebra $A$ is  $\Hom_{\Fone}(\Fone[M], A)$. Because $A$ is a multiplicative monoid, these points can be multiplied pointwise.

Over the enriched base $\Ftwo$, a fundamental subtlety arises. A rational point over an $\Ftwo$-algebra $A$ is a homomorphism $\phi \in \Hom_{\Ftwo}(\Ftwo[M], A)$, which, by definition, must preserve the canonical sign: $\phi(\epsilon) = -1_A$. If one were to take the naive pointwise product of two such points, the sign would be destroyed:
\begin{equation*}
(\phi_1 \cdot \phi_2)(\epsilon) = \phi_1(\epsilon) \cdot \phi_2(\epsilon) = (-1_A)(-1_A) = 1_A \neq -1_A.
\end{equation*}
This indicates that the naive pointwise product drops out of the category of $\Ftwo$-algebras. The rigorous resolution relies on the Lydakis-Day tensor product and the intrinsic coproduct of the spherical algebra.

\begin{definition}
For a spherical $\Ftwo$-algebra $\Ftwo[M]$, the diagonal map $m \mapsto m \wedge m$ on the spatial monoid defines an $\Ftwo$-linear coproduct:
\begin{equation*}
\Delta: \Ftwo[M] \longrightarrow \Ftwo[M] \otimes_{\Ftwo} \Ftwo[M], \qquad \Delta(m) = m \otimes m \quad \text{for all } m \in M.
\end{equation*}
\end{definition}

We input this coproduct to define the correct multiplication of rational points.

\begin{proposition}
Let $\Ftwo[M]$ be a spherical algebra and $A$ an $\Ftwo$-algebra. The set of points $X(A) = \Hom_{\Ftwo}(\Ftwo[M], A)$ canonically forms a monoid under the convolution product:
\begin{equation}\label{conv}
\phi_1 \ast \phi_2 := (\phi_1 \otimes \phi_2) \circ \Delta.
\end{equation}
If the non-zero elements of $M$ form a group, $X(A)$ canonically forms a group.
\end{proposition}

\begin{proof}
We need to verify that $\phi_1 \ast \phi_2$ is a valid $\Ftwo$-algebra homomorphism, which reduces to checking that it preserves the canonical sign $\epsilon$. Because $\Delta$ is $\Ftwo$-linear, it absorbs the sign via the smash product relations: $\Delta(\epsilon) = \epsilon \otimes 1$. Evaluating the convolution yields:
\begin{equation*}
(\phi_1 \ast \phi_2)(\epsilon) = (\phi_1 \otimes \phi_2)(\epsilon \otimes 1) = \phi_1(\epsilon)\phi_2(1) = (-1_A) \cdot 1_A = -1_A.
\end{equation*}
On the spatial elements $m \in M$, the convolution yields the expected multiplication:
\begin{equation*}
(\phi_1 \ast \phi_2)(m) = (\phi_1 \otimes \phi_2)(m \otimes m) = \phi_1(m)\phi_2(m).
\end{equation*}
Thus, the Lydakis-Day tensor product natively ``protects'' the base ring, thus the space of rational points inherits a canonical monoid structure.
\end{proof}

\section{The Unfolded Line and the Imaginary Generator}

To describe the local geometry at the archimedean place, we begin with the
scheme-theoretic construction of the projective line $\P^1_{\Fone}$ over the
absolute base \cite{CC10}. Motivated by extending the action of the geometric involution $\alpha$ (see \eqref{alpha})  at the sheaf level (as explained in the introduction), we  first extend scalars from $\Fone$ to the
quadratic absolute base $\Ftwo$ as defined in~\eqref{F2}.

The underlying topological space $X$ of $\P^1_{\Ftwo}$ consists of two closed points, denoted
$+$ and $-$, together with a unique generic point $\eta$. The nontrivial open
subsets are
\begin{equation}\label{opens1}
\Omega_+=\{+,\eta\},\qquad
\Omega_-=\{-,\eta\},\qquad \Omega_+\cap \Omega_- = \Omega=\{\eta\}.
\end{equation}
The structure sheaf is given on these affine open subsets as follows:
\begin{equation*}
\begin{aligned}
\mathcal O_X(\Omega_+) &= \Ftwo[T^{\Z_{\ge0}}],\\
\mathcal O_X(\Omega_-) &= \Ftwo[T^{\Z_{\le0}}],\\
\mathcal O_X(\Omega)   &= \Ftwo[T^{\Z}].
\end{aligned}
\end{equation*}

These local descriptions are entirely consistent with the local structure of
the $\Fone$-arithmetic site $(\nato,\Fone[T])$, where the stalk attached to an \emph{ordered} rank-one abelian
group $(H,H_+)$ is the spherical algebra $\Fone[T^{H_+}]$. The novelty in the archimedean case is that neither of these two
orderings is intrinsically preferred. This observation is the starting point
for the equivariant construction developed in this section.

The two affine open subsets $\Omega_+$ and $\Omega_-$ correspond to the two
possible orderings of the rank-one abelian group $\Z$, while their generic
intersection
$
\Omega=\{\eta\}
$
corresponds to the unordered group $\Z$ itself. The restriction morphisms are
induced by the canonical inclusions
\[
\rho_+:\Ftwo[T^{\Z_{\ge 0}}]\hookrightarrow \Ftwo[T^{\Z}]
\qquad\text{and}\qquad
\rho_-:\Ftwo[T^{\Z_{\le 0}}]\hookrightarrow \Ftwo[T^{\Z}].
\]
Thus, the projective line $\P^1_{\Ftwo}$ provides the canonical geometric
object obtained by gluing the two ordered models of $\Z$ along their common
generic localization.\vspace{.05in}

The extension of scalars from $\Fone$ to $\Ftwo$ preserves the expected geometry after
evaluation on a field, as shown by the following statement

\begin{proposition}\label{prop:points_P1}
Let $K$ be a field. The set of $K$-valued points of the absolute projective
line $\P^1_{\Ftwo}$ is canonically identified with the classical   projective line $\P^1(K)$.
\end{proposition}

\begin{proof}
A $K$-point of $\P^1_{\Ftwo}$ is given by an element of the colimit of the sets $$\Hom(\mathcal O_X(U),HK)$$ as $U$ varies among the affine open sets. 
The projective line $\P^1_{\Ftwo}$ is covered by the three affine open subsets \eqref{opens1}.

By the universal property of the free $\Ftwo$-algebra
$\Ftwo[T^{\Z_{\ge0}}]$, every $\Ftwo$-algebra homomorphism is uniquely
determined by the image of $T$, which may be chosen arbitrarily in $K$. Hence there is a canonical
identification
\[
\Hom(\mathcal O_X(\Omega_+),HK) \cong K.
\]
Similarly,
$
\Hom(\mathcal O_X(\Omega_-),HK) \cong K$, and $\Hom(\mathcal O_X(\Omega),HK) \cong K^\times$.
Under these identifications the restriction maps $\rho_\pm$ act as follows 
\[
\rho_+^*: K^\times \to K, ~\rho_+^*(x)=x, \qquad \rho_-^*(x): K^\times \to K,~\rho_-^*(x)=x^{-1}.
\]
Hence the set $\P^1_{\Ftwo}(K)$ is obtained by gluing two copies of $K$ along $K^\times$ using the above inclusions, 
 yielding
precisely the set $\P^1(K)$ of the classical projective line.
\end{proof}

The crucial new step is now the \emph{adjunction} of a \emph{Bombelli generator},
namely a formal imaginary element $J$ satisfying
\[
J^2=\epsilon.
\]
Accordingly, the structure sheaf  of $\P^1_{\Ftwo}$ gets enriched by adjoining local 
generators $J_+$ and $J_-$ to the algebras of  the two affine charts, together with a generic
generator $J$ on their overlap. It is defined as follows:
\begin{equation}\label{opens}
\begin{aligned}
\mathcal{O}(\Omega_+) &= \Ftwo[T^{\Z_{\ge0}},J_+]/(J_+^2=\epsilon),\\
\mathcal{O}(\Omega_-) &= \Ftwo[T^{\Z_{\le0}},J_-]/(J_-^2=\epsilon),\\
\mathcal{O}(\Omega)   &= \Ftwo[T^{\Z},J]/(J^2=\epsilon).
\end{aligned}
\end{equation}

The gluing  of sections is determined by the restriction
morphisms $\rho_{\pm}$ to the generic overlap $\Omega$. The coordinate $T$ restricts
canonically, whereas the Bombelli generators satisfy the twisted restriction
rules
\begin{equation}\label{eq:restriction_maps}
J_+\longmapsto J,
\qquad
J_-\longmapsto\epsilon J.
\end{equation}
The factor $\epsilon$ that was originally introduced as responsible for the sign inversion is also ensuring, in the above formula,  that the two choices of $\sqrt{-1}$ are both available.\vspace{.05in}

We denote by
\begin{equation}\label{X}
\mathcal X:=(\P^1_{\Ftwo},\mathcal O)
\end{equation}
the $\Ftwo$-projective line endowed with the enriched structure sheaf constructed
above.\vspace{.05in}

The evaluation of $\mathcal X$ on a field $K$ is governed entirely by the
solvability of the equation
\[
x^2=-1
\]
over $K$.

\begin{proposition}\label{prop:points_X_K}
Let $K$ be a field. The set of $K$-valued points of $\mathcal X$ reflects the arithmetic of the equation
$
x^2=-1
$
over the field $K$:
\begin{enumerate}
\item If $K$ contains no square root of $-1$ (for example $K=\R$), then
\[
\mathcal X(K)=\varnothing.
\]
\item If $\operatorname{char}(K)\neq2$ and $K$ contains the two square roots
$\pm i_K$ of $-1$ (for example $K=\C$), then
\[
\mathcal X(K)\cong \P^1(K)\sqcup\P^1(K).
\]
\item If $\operatorname{char}(K)=2$, then
\[
\mathcal X(K)\cong\P^1(K).
\]
\end{enumerate}
\end{proposition}

\begin{proof}
Since $\mathcal X$ is obtained by gluing the affine charts $\Omega_+$ and $\Omega_-$
along the generic open subset $\Omega$, its set of $K$-valued points is
obtained by gluing the corresponding affine sets of $K$-points.
Every $\Ftwo$-algebra homomorphism
\[
\phi:\mathcal O(U)\longrightarrow HK
\]
satisfies
$
\phi(\epsilon)=-1_K$. 
Consequently, the defining relation
$
J^2=\epsilon
$
is evaluated in $K$ as
$
j^2=-1_K$.
A $K$-valued point of the affine chart $\Omega_+$ is therefore determined by
a pair:
\[
(t,j_+)\in K\times K
\]
satisfying
$
j_+^2=-1_K$. 
Similarly, a point of $\Omega_-$ is determined by a pair
\[
(s,j_-)\in K\times K
\]
with
$
j_-^2=-1_K$.

If $K$ contains no square root of $-1$, then neither affine chart possesses
$K$-points. Hence
\[
\Omega_+(K)=\varnothing=\Omega_-(K),
\]
and therefore
$
\mathcal X(K)=\varnothing$.

Assume now that $\operatorname{char}(K)\neq2$ and that
$
-1\in (K^\times)^2$. 
Then $K$ contains exactly two square roots of $-1$, denoted
$\pm i_K$. 
Accordingly,
\[
\Omega_+(K)=K\times\{i_K,-i_K\},
\]
\[
\Omega_-(K)=K\times\{i_K,-i_K\},
\]
and
\[
\Omega(K)=K^\times\times\{i_K,-i_K\}.
\]
The restriction maps
\eqref{eq:restriction_maps}
identify the two affine charts through
\begin{equation}\label{eq:restriction_maps1}
T\longmapsto T^{-1},
\qquad
J_+\longmapsto J,
\qquad
J_-\longmapsto\epsilon J.
\end{equation}
Since $\phi(\epsilon)=-1_K$, the last relation becomes
$
j_-= -\,j_+$. 
Hence the gluing identifies
\[
(t,j)\sim(t^{-1},-j),
\]
Consequently,  the set of $K$-valued points decomposes canonically into two disjoint
copies of $\P^1(K)$.
\[
\mathcal X(K)\cong
\P^1(K)\sqcup\P^1(K),
\]
where the first copy is obtained by gluing
$
K\times\{i_K\}_+$
with
$K\times\{-i_K\}_-$, 
and the second by gluing
$
K\times\{-i_K\}_+$ with $K\times\{i_K\}_-$.

Finally, if $\operatorname{char}(K)=2$, then: 
$
-1_K=1_K$, 
and the equation
$
j^2=-1_K
$
has the unique solution: $
j=1$. 
The twist by $\epsilon$ therefore becomes trivial, and the two affine charts
glue in the usual way through
\[
t\longleftrightarrow t^{-1},
\]
recovering a single copy of the projective line:
$
X(K)\cong\P^1(K)$.\vspace{.05in}

Thus the doubled projective line over fields containing $\sqrt{-1}$, and the
absence of real points, arise as direct consequences of the signed relation
$J^2=\epsilon$.
\end{proof}

\section{The Geometric Symmetry and the Equivariant Topos}

The absolute projective line $\mathcal X=(\P^1_{\Ftwo},\mathcal O)$ introduced in \eqref{X} carries a canonical involutive symmetry, reflecting the
absence of a distinguished ordering of the rank-one abelian group $\Z$. This
symmetry exchanges the two affine charts $\Omega_\pm$ and \emph{simultaneously} implements the
signed inversion underlying the archimedean geometry.

\begin{definition}\label{def:alpha_symmetry}
Let
$
\alpha:\mathcal X\longrightarrow \mathcal X
$
be the involutive automorphism defined on the underlying topological space by
interchanging the two closed points,
\[
\alpha(+)= -,\qquad
\alpha(-)= +,
\]
while fixing the generic point
$
\alpha(\eta)=\eta$.

On the generic affine open subset $\Omega$, the induced automorphism of the
structure sheaf
\[
\alpha^*:\mathcal O(\Omega)\longrightarrow\mathcal O(\Omega)
\]
is determined by
\[
\alpha^*(T)=\epsilon T^{-1},
\qquad
\alpha^*(J)=J^{-1}=\epsilon J.
\]
The first formula realizes the signed inversion of the spatial coordinate,
while the second exchanges the two branches determined by the relation
$J^2=\epsilon$. (Note that
$
J(\epsilon J)=\epsilon J^2=\epsilon^2=1$, 
so that
$
J^{-1}=\epsilon J$.)

The generic automorphism extends uniquely to isomorphisms of the affine
charts
\[
\alpha^*_\pm:
\mathcal O(\Omega_\mp)
\longrightarrow
\mathcal O(\Omega_\pm),
\]
compatible with the restriction maps
\eqref{eq:restriction_maps}. Explicitly,
\[
\alpha_+^*(T^{-1})=\epsilon T,
\qquad
\alpha_+^*(J_-)=J_+,
\]
and
\[
\alpha_-^*(T)=\epsilon T^{-1},
\qquad
\alpha_-^*(J_+)=J_-.
\]
\end{definition}

The compatibility of the local and generic actions follows immediately from
the definitions. Indeed, restricting $\alpha_+^*(J_-)$ to the generic open
subset $\Omega$ gives
$
J$, 
while
\[
\alpha^*(\rho_-(J_-))
=
\alpha^*(\epsilon J)
=
\epsilon(\epsilon J)
=
J.
\]
The remaining compatibility relations are verified in the same way.

The action of $\alpha$ on the imaginary generator,
\begin{equation}\label{epsi}
\alpha^*(J)=J^{-1}=\epsilon J,
\end{equation}
is uniquely determined by the geometry. \vspace{.05in}

This naturally leads to the main geometric object at the
archimedean place, namely  \emph{the equivariant topos} endowed with structure sheaf:
\begin{equation}\label{eqtop}
\minf:=[\mathcal X/\langle\alpha\rangle].
\end{equation}

If $K$ is a field as in 
Proposition~\ref{prop:points_X_K}, case \emph{2}., then $\mathcal X(K)\cong \P^1(K)\sqcup\P^1(K)$ and  the involution $\alpha$
interchanges the two copies of $\P^1(K)$.
Indeed, if a $K$-valued point $\phi$ belongs to the first copy, then the equalities
\[
(\alpha\cdot\phi)(J_-)
=
\phi(\alpha_+^*(J_-))
=
\phi(J_+),
\]
show that  $\alpha\cdot\phi$ belongs to the second copy, and conversely.
Passing to the quotient therefore identifies the two copies canonically:
\[
\minf(K)
=
[\mathcal X(K)/\langle\alpha\rangle]
\cong
\P^1(K).
\]
 Arithmetically,
\eqref{epsi} ensures that the $\alpha$-invariant part of the generic stalk is precisely
the base $\Ftwo$, preventing the additional archimedean structure to
appear at the finite primes.\vspace{.03in}

The underlying topological space of 
$
\minf
$
consists of  two points. The involution $\alpha$ acts freely on the two
closed points $\{+,-\}$, identifying them with a single closed point $\gamma$. The
generic point $\eta$ is fixed, but acquires isotropy group $\Z/2\Z$, thereby
becoming an orbifold point. Geometrically, the quotient folds the absolute
projective line by identifying the two orderings of $\Z$.

\section{Complex Conjugation and the Twistor Real Structure}

To determine the real structure of the equivariant topos
$\minf$ (see \eqref{eqtop}) at the archimedean place, we first analyze the natural
action of the Galois group
$
\Gal(\C/\R)=\{1,c\}
$
on the complex points of $\mathcal X$ (see \eqref{X}). Every automorphism of the
ground field acts on the set of field-valued points by post-composition.

Let
$
c:\C\longrightarrow\C
$
denote  the complex conjugation. If
$
\phi\in \mathcal X(\C),
$
its conjugate point is defined by
\[
c\cdot\phi:=c\circ\phi.
\]
Recall the canonical decomposition
\begin{equation}\label{twop1}
\mathcal X(\C)\cong
\P^1(\C)_1\sqcup\P^1(\C)_2,
\end{equation}
where $\P^1(\C)_1$ is obtained by gluing
\[
\C\times\{i\}_+
\qquad\text{with}\qquad
\C\times\{-i\}_-,
\]
while $\P^1(\C)_2$ is obtained by gluing
\[
\C\times\{-i\}_+
\qquad\text{with}\qquad
\C\times\{i\}_-.
\]
Let
$
\phi\in\C\times\{i\}_+\subset\Omega_+(\C)$, 
so that
\[
\phi(T)=z,
\qquad
\phi(J_+)=i.
\]
Then
\begin{align}
(c\cdot\phi)(T)&=\overline z,\label{phip}\\
(c\cdot\phi)(J_+)&=-i.\nonumber
\end{align}
Hence
\[
c\cdot\phi\in
\C\times\{-i\}_+\subset\Omega_+(\C),
\]
and therefore belongs to the second copy
$\P^1(\C)_2$.
More generally, complex conjugation acts by
\[
z\times \{\pm i\}_\pm
\longmapsto
\bar z\times \{-\pm i\}_\pm,
\]
thereby exchanging the two copies of
$\P^1(\C)$ appearing in~\eqref{twop1}.
Passing to the quotient
\[
\minf(\C)
=
[\mathcal X(\C)/\langle\alpha\rangle],
\]
the two copies are identified by the geometric involution $\alpha$, while
complex conjugation descends to an anti-holomorphic involution of the
resulting projective line. 
The geometric involution $\alpha$ and the arithmetic involution $c$ act
independently on $\mathcal X(\C)$ and commute with one another. Consequently, the
Galois action descends to the quotient
$\minf(\C)$, endowing the latter with its canonical real
structure.

\begin{theorem}
\label{thm:twistor_space_scheme}
The canonical action of the complex conjugation
$
c\in\Gal(\C/\R)
$
on the $\C$-valued points of the equivariant topos
$\minf$ induces the twistor real structure. More precisely,
on the resulting complex projective line the induced real structure is the
fixed-point-free anti-holomorphic involution:
\[
z\longmapsto-\frac1{\bar z}.
\]
\end{theorem}

\begin{proof}
 Since $\alpha$ interchanges the two copies of $\P^1(\C)$ in   the decomposition \eqref{twop1}, the orbit of the action of the symmetries on the complex points of the unfolded space $\mathcal X(\C)$ are parametrized by their intersection with the first copy $\P^1(\C)_1$. 
The complex conjugation replaces $z\times \{\pm i\}_\pm$ by $\bar z\times \{-\pm i\}_\pm$. The first copy $\P^1(\C)_1$ is obtained by gluing
$
\C\times\{i\}_+$
with
$\C\times\{-i\}_-$, using the canonical map
\[
\C^\times \to \left(\C\times\{i\}_+\right)\times \left(\C\times\{-i\}_-\right), \  \ z\mapsto ((z,i),(z^{-1},-i)).
\]
corresponding to the restriction maps \eqref{eq:restriction_maps1}. We can thus parametrize the points of $\P^1(\C)_1$ by the bijection\footnote{This bijection extends \emph{by continuity} to $\C\cup\{\infty\}$ the parametrization of points belonging to the Zariski dense open set $\Omega(\C)$.}
\[
P:\C\cup \{\infty\}\to \P^1(\C)_1, \  z\mapsto z\times \{ i\}_+, \ \infty  \mapsto 0 \times  \{ -i\}_-.
\]
For $z\in \C^\times$, the point $P(z)$ corresponds to the homomorphism 
\[
\phi: \mathcal O(\Omega)\to H\C, \ \phi(T)=z, \ \phi(J)=i.
\]
For $z\in \C^\times$, the orbit of the point $P(z)$ under the geometric symmetry $\alpha$ is (using Definition \ref{def:alpha_symmetry}) the equivalence class:
\[
[P] = \big\{ P, \, \alpha(P) \big\} = \big\{ (z, i), \, (-z^{-1}, -i) \big\}.
\]
Applying standard complex conjugation $c$ to this orbit yields the conjugate equivalence class:
\[
c([P]) = \big\{ (\bar{z}, -i), \, c(-z^{-1}, -i) \big\} = \big\{ (\bar{z}, -i), \, (-\bar{z}^{-1}, i) \big\}.
\]
To determine the induced action on the quotient space (which we may parameterize by the map $P$), we select the representative of the conjugate orbit that lies on $\P^1(\C)_1$. This representative is exactly $(-\bar{z}^{-1}, i)$. 

Thus, the induced action on the spatial coordinate is $z \mapsto -1/\bar{z}$, and it swaps the points $0$ and $\infty$ (as can be seen by \emph{continuity}).
\end{proof}

This demonstrates that the twistor real structure  is the automatic, functorial consequence of standard complex conjugation acting on the geometrically folded absolute projective line. The lack of real points at infinity is  guaranteed by the spatial sign twist $\epsilon$ inherent in the Weyl symmetry $\alpha$.

\section{Frobenius and Twistor Equivariance}

The intrinsic Hopf structure of spherical $\Ftwo$-algebras established in Section~\ref{sect2} is the  mechanism that generates the geometry of the archimedean twistor line. 

First, consider the absolute Frobenius endomorphism $\frob_n(x) = x^n$ acting on a spherical algebra. Under the intrinsic convolution product \eqref{conv}, the $n$-th power Frobenius of a $A$-rational point $\phi \in \Hom_{\Ftwo}(\Ftwo[M], A)$ is: $$\phi^n(m) = \phi(m^n) = \phi(\frob_n(m)).$$ 
Thus,  $\frob_n$ is the canonical geometric $n$-th power map in the intrinsic monoid of points. This shows why the Frobenius manifests dynamically as $z \mapsto z^n$ on the spatial coordinate of the twistor line, and also why it \emph{geometrically implements the Adams operations} $\psi^n$.\vspace{.03in}

Second,  this Hopf structure provides the absolute algebraic origin for non-abelian Hodge theory. The generic open set $\Omega$ of the absolute curve (\eqref{opens}) has the spatial core $\Ftwo[T^\Z]$. Because the non-zero elements of $M = \Z$ form a group, its geometric complex points $\Omega(\C)$ canonically form the multiplicative group $\C^\times$. This group acts on the entire twistor line by scaling the spatial coordinate via the coproduct mechanism.

In Simpson's twistor construction \cite{Si}, a real Hodge structure is encoded as a \sloppy $\C^\times$-equivariant vector bundle on $\mathbb{P}^1(\C)$, where the real structure $\sigma(z) = -1/\bar{z}$ must satisfy the twisted equivariance:
\begin{equation*}
\sigma(\lambda z) = \frac{-1}{\overline{\lambda z}} = \bar{\lambda}^{-1} \sigma(z) \qquad \text{for } \lambda \in \C^\times.
\end{equation*}
In our absolute framework, the required $\C^\times$-action is not imposed externally; rather it is the canonical action of the generic absolute stalk $\Omega(\C)$ acting via the $\Ftwo$-coproduct as in \S\,\ref{sect2.2}. The compatibility with the geometric symmetry $\alpha^*(T) = \epsilon T^{-1}$ perfectly  recovers this twisted equivariance. 

Consequently, the absolute geometry of $\Ftwo$ dynamically generates the entire architecture of real Hodge structures: the intrinsic Lydakis-Day coproduct yields the $\C^\times$-equivariance, the geometric symmetry $\alpha$ yields the conjugate filtration, and finally  the odd Frobenius $\frob_n$ implements the Adams operations.

\section{The Global Absolute Curve and the Amalgam}\label{sec:global_curve}

With the archimedean component $\minf$ rigorously defined, we can now move to the construction of the global absolute curve $\bar{\mathscr{C}}=(\spzb)_{\Ftwo}$ by amalgamating the finite and archimedean components.\vspace{.05in} 

Let $\mathscr{C}_{\text{fin}}:=(\Spec \Z)_{\Ftwo}$ denote the finite absolute curve, after extension of scalars from $\Fone$ to $\Ftwo$. Topologically, it consists of a countable set of closed points (the finite primes $\mathcal{P}$) and a single generic point $\eta_{\text{fin}}$. The stalk of the structure sheaf at the generic point is exactly the base $\Ftwo$.

On the other hand, the archimedean component $\minf$ consists topologically of a single closed point $\infty$ (the folded points $\{+, -\}$) and an orbifold generic point $\eta_\infty$. Crucially,  the $\alpha$-invariant stalk at $\eta_\infty$ reduces strictly to $\Ftwo$. \vspace{.03in}

Because the stalks at the generic points of both spaces are canonically isomorphic to $\Ftwo$, we can define the global curve as their topological and sheaf-theoretic amalgam.

\begin{definition}\label{def:absolute_curve0}
The global absolute curve $$\bar{\mathscr{C}}:=(\spzb)_{\Ftwo}$$ is the pushout (amalgam) of $\mathscr{C}_{\text{fin}}$ and $\minf$ along their generic points. 
\begin{itemize}
    \item  The underlying topological space $|\bar{\mathscr{C}}|$ is the colimit identifying $\eta_{\text{fin}}$ with $\eta_\infty$:
    \begin{equation*}
|\bar{\mathscr{C}}| = |\mathscr{C}_{\text{fin}}| \coprod_{\eta_{\text{fin}} \sim \eta_\infty} |\minf|.
\end{equation*}
    It consists of the closed points $\mathcal{P} \cup \{\infty\}$ and a single, unified generic point $\eta_{\text{global}}$.
    
    \item The structure sheaf $\mathcal{O}_{\bar{\mathscr{C}}}$ is defined by Artin gluing along the generic point. For any Zariski open set $U$ of $\bar{\mathscr{C}}$, the sections are given by the fiber product of the sections on the finite and archimedean components over the generic stalk $\Ftwo$:
    \begin{equation*}
\mathcal{O}_{\bar{\mathscr{C}}}(U) = \mathcal{O}_{\text{fin}}(U \cap \mathscr{C}_{\text{fin}}) \times_{\Ftwo} \mathcal{O}_{\infty}(U \cap \minf).
 \end{equation*}
\end{itemize}
\end{definition}

This amalgamation is the absolute algebraic realization of the  geometry of $\spzb$. \vspace{.03in}

In classical Arakelov theory, the archimedean place is added formally, as it lacks a true geometric stalk. Here, $\minf$ provides the required geometric stalk. The twistor real structure unifies the complex spheres over the closed point $\infty$, while the orbifold symmetry collapses the generic stalk to $\Ftwo$,  matching the finite generic point and allowing the global curve to glue together.

\subsection{The Absolute Curve in the Odd Arithmetic Topos}

To properly define the global structure of the absolute curve, we must specify the ambient category in which it lives. While in classical algebraic geometry, a curve is a locally ringed space, in absolute geometry over $\Fone$ or $\Ftwo$, the foundational ambient category is the arithmetic topos . The main reason for this is the existence of a canonical, absolute geometric operation: the Frobenius endomorphism.

The objects of the arithmetic topos are sets (sheaves of sets) equipped with an action of the multiplicative monoid of integers. When working over $\Fone$, the relevant topos is $\widehat{\N^\times}$, and the absolute Frobenius endomorphism $\frob_n(x) = x^n$ acts for all $n \in \N^\times$. However, passing to the enriched base $\Ftwo$ introduces an arithmetic subtlety. For $\frob_n$ to be a valid endomorphism of an $\Ftwo$-algebra, it must fix this base algebra pointwise; in particular, it must fix the canonical sign $\epsilon = -1$. This forces a restriction to the submonoid of odd integers, $\N^\times_{\text{odd}}$.\vspace{.03in}

We therefore define the absolute curve not merely as a static ringed space, but as a \emph{dynamical geometric object} internal to this odd arithmetic topos.

\begin{definition}\label{def:absolute_curve}
The global absolute curve $\bar{\mathscr{C}}$ is an object internal to the odd arithmetic topos $\widehat{\N^\times_{\text{odd}}}$. It is defined by the following data:
\begin{itemize}
    \item \textbf{Topological Space:} The underlying space $|\bar{\mathscr{C}}|$ is the amalgam of the finite primes $\mathcal{P}$ and the archimedean closed point $\infty$, glued over a dense generic point $\eta_{\text{global}}$. The monoid $\N^\times_{\text{odd}}$ acts trivially on this  space.
    \item \textbf{Structure Sheaf:} The structure sheaf $\mathcal{O}_{\bar{\mathscr{C}}}$ is a sheaf of $\Ftwo$-algebras equipped with the canonical absolute Frobenius action. For any open set $U$ and any section $s \in \mathcal{O}_{\bar{\mathscr{C}}}(U)$, the monoid $\N^\times_{\text{odd}}$ acts via the endomorphisms:
    \begin{equation*}
    \frob_n(s) = s^n \qquad \text{for all } n \in \N^\times_{\text{odd}}.
    \end{equation*}
\end{itemize}
\end{definition}

By defining the curve internally to the topos, the notion of a global section is fundamentally upgraded. In any topos, the space of global sections of an object $Y$ is defined categorically as the set of morphisms from the terminal object $\mathbf{1}$ to $Y$. In this case:
\begin{equation*}
\Gamma(Y) := \Hom_{\widehat{\N^\times_{\text{odd}}}}(\mathbf{1}, Y).
\end{equation*}
The terminal object $\mathbf{1}$ is the one-point set endowed with the trivial monoid action. Consequently, a morphism $\mathbf{1} \to Y$ corresponds exactly to selecting an element of $Y$ that is equivariant with respect to the trivial action—that is, an element invariant under the monoid. 

Applying this categorical rule to the structure sheaf of the absolute curve, a true global section is not merely a compatible family of local sections, but a family of local sections that are \emph{simultaneous fixed points} of the absolute Frobenius action. This topos-theoretic requirement provides the exact mechanism for global rigidity.

\begin{proposition}\label{prop:global_sections_frobenius}
The space of global sections of the absolute compactified curve $\bar{\mathscr{C}}$ is precisely the base ring $\Ftwo$:
\begin{equation*}
\Gamma(\bar{\mathscr{C}}, \mathcal{O}_{\bar{\mathscr{C}}}) = \Ftwo.
\end{equation*}
\end{proposition}

\begin{proof}
Let $s \in \Gamma(\bar{\mathscr{C}}, \mathcal{O}_{\bar{\mathscr{C}}})$ be a global section. Topologically, $s$ defines a local germ $s_v \in \mathcal{O}_v$ at every closed point $v \in \mathcal{P} \cup \{\infty\}$, as well as a generic germ $s_\eta \in \mathcal{O}_\eta$ at the dense generic point $\eta_{\text{global}}$. 

The sheaf compatibility condition requires that for every closed point $v$, the canonical localization map $\rho_v: \mathcal{O}_v \to \mathcal{O}_\eta$ must send the local germ to the generic one:
\begin{equation}\label{eq:sheaf_compat}
\rho_v(s_v) = s_\eta.
\end{equation}

Because the curve is defined internally to the arithmetic topos, a global section is a morphism from the terminal object, which means it must be a fixed point of the odd absolute Frobenius action. Thus, at every place $v$, we  have $\frob_n(s_v) = s_v^n = s_v$ for all $n \in \N^\times_{\text{odd}}$.

As shown previously, at the archimedean place, the local stalk is modeled by $$\Ftwo[T^{\Z_{\geq 0}}, J_+] / (J_+^2 = \epsilon).$$ The condition $s_\infty^n = s_\infty$ for all odd $n \geq 3$ forces the spatial and imaginary components to vanish (since $T^n \neq T$ and $J_+^3 = -J_+ \neq J_+$). Therefore, the Frobenius fixed points of the stalk reduce strictly to the base ring: $s_\infty \in \Ftwo$. By identical reasoning at the finite primes, $s_p \in \Ftwo$ for all $p \in \mathcal{P}$.

We now apply the sheaf compatibility condition \eqref{eq:sheaf_compat}. Because the localization maps $\rho_v$ are $\Ftwo$-algebra homomorphisms, they act as the identity on the base ring $\Ftwo$. Since $s_v \in \Ftwo$, we have:
\begin{equation*}
s_v = \rho_v(s_v) = s_\eta \quad \text{for all } v \in \mathcal{P} \cup \{\infty\}.
\end{equation*}
This forces $s_v$ to be the exact same constant in $\Ftwo$ at every place, uniquely determined by the generic germ $s_\eta$. Thus, the global sections are exactly the constant sections $\Ftwo$.
\end{proof}

\begin{remark}
While the absolute curve is defined via the internal topos perspective, evaluating this structure on geometric points (e.g., over $\C$) naturally recovers the noncommutative dynamical perspective. The canonical Frobenius action on the abstract $\Ftwo$-algebra induces the branched dynamical system $z \mapsto z^n$ and $z \mapsto -1/z^n$ on the complex twistor line. The noncommutative crossed product space $\bar{\mathscr{C}}(\C) \rtimes \N^\times_{\text{odd}}$ is therefore a derived geometric consequence of the canonical internal topos structure defined above.
\end{remark}

\subsection{The Odd Monoid Action on Twistor Line and  Ramification at 2}

The restriction of the absolute Frobenius to the odd monoid $\N^\times_{\text{odd}}$ has an important geometric manifestation on the archimedean twistor line. Because the $\Ftwo$-enrichment adjoins an imaginary generator $J$ satisfying $J^2 = \epsilon = -1$, it acts as the absolute geometric avatar of the Gaussian integers $\Z[i]$. 

The prime $2$ is uniquely ramified in $\Z[i]$, while the behavior of odd primes is governed by the non-trivial Dirichlet character modulo $4$. This exact arithmetic structure emerges dynamically on the complex points of the absolute curve. The exclusion of the even integers from the monoid action corresponds precisely to the ramification of the prime 2, while the odd monoid induces a branched dynamical system on the unified twistor line.\vspace{.05in}

Let  $\minf(\C) \cong \mathbb{P}^1(\C)$  be the unified complex twistor line, obtained by quotienting the $\Ftwo$-scheme $\mathcal X(\C)$ (see \eqref{X}) by the geometric symmetry $\alpha$.

\begin{proposition}\label{prop:twistor_frobenius}
 The odd arithmetic monoid $\N^\times_{\text{odd}}$ acts on $\minf(\C)$ via the dynamical system:
\begin{equation*}
\frob_n(z) := 
\begin{cases} 
z^n & \text{if } n \equiv 1 \pmod 4, \\
-\frac{1}{z^n} & \text{if } n \equiv 3 \pmod 4.
\end{cases}
\end{equation*}
Furthermore, this action strictly commutes with the quaternionic real structure $$\sigma(z) = -1/\bar{z}.$$
\end{proposition}

\begin{proof}
On the unquotiented space $\mathcal X(\C) = \mathbb{P}^1(\C)_+ \sqcup \mathbb{P}^1(\C)_-$, a complex point is determined by a spatial coordinate $z \in \C \cup \{\infty\}$ and an imaginary coordinate $j \in \{i, -i\}$. The algebraic Frobenius $\frob_n(x) = x^n$ acts on these coordinates as:
\[
\frob_n(z, j) = (z^n, j^n) = \big(z^n, (-1)^{\frac{n-1}{2}} j\big).
\]
If $n \equiv 1 \pmod 4$, the imaginary coordinate is preserved, meaning $\frob_n$ maps each sphere to itself via $z \mapsto z^n$. 

If $n \equiv 3 \pmod 4$, the imaginary coordinate is inverted ($j \mapsto -j$), meaning the point is mapped to the opposite sphere. To evaluate this action on the quotient space $\minf(\C)$, we must identify the target point with its representative on the original sphere by applying the geometric symmetry $\alpha$. By Definition \ref{def:alpha_symmetry}, $\alpha^*(T) = \epsilon T^{-1}$, which induces the geometric map $\alpha(w) = -1/w$ on the complex coordinate. Composing the jump with this identification yields the effective action:
\[
z \longmapsto z^n \stackrel{\alpha}{\longmapsto} -\frac{1}{z^n}.
\]
To verify compatibility with the twistor real structure $\sigma(z) = -1/\bar{z}$, we check both cases. For $n \equiv 1 \pmod 4$:
\[
\sigma(z^n) = -\frac{1}{\overline{z^n}} \qquad \text{and} \qquad (\sigma(z))^n = \left(-\frac{1}{\bar{z}}\right)^n = -\frac{1}{\overline{z^n}} \quad \text{(since $n$ is odd)}.
\]
For $n \equiv 3 \pmod 4$:
\[
\sigma\left(-\frac{1}{z^n}\right) = -\frac{1}{\overline{(-1/z^n)}} = \bar{z}^n \qquad \text{and} \qquad -\frac{1}{(\sigma(z))^n} = -\frac{1}{(-1/\bar{z})^n} = \bar{z}^n.
\]
Thus, the branched dynamical system perfectly preserves the quaternionic structure.
\end{proof}

This proposition shows that the absolute Frobenius action on the archimedean component is not merely a formal algebraic operation, but a highly structured geometric endomorphism. The fact that the map $z \mapsto -1/z^n$ naturally emerges for $n \equiv 3 \pmod 4$ shows that the $\Ftwo$-twistor sheaf intrinsically "knows" about complex conjugation and the non-trivial topology of the archimedean place.


\begin{thebibliography}{9}

\bibitem{CC1} A. Connes, C. Consani, \emph{\href{https://arxiv.org/pdf/2606.06604}{On the absolute geometry of $\overline{\text{Spec}\,\mathbf Z}$ and the Fargues-Fontaine curve}}, Preprint (2026).  

\bibitem{CC2} A. Connes, C. Consani, \emph{\href{https://arxiv.org/pdf/2602.15941}{On the Jacobian of $\overline{\text{Spec}\,\mathbf Z}$}}, to appear in Journal of Noncommutative Geometry. Preprint (2026), 

\bibitem{CC3} A. Connes, C. Consani, \emph{Riemann-Roch for the ring $\mathbf{Z}$}, C. R. Math. Acad. Sci. Paris \textbf{362} (2024), 229-235.

\bibitem{CC4} A. Connes,  C. Consani, \emph{Cyclic theory and the pericyclic category}, in  Cyclic cohomology at 40: achievements and future prospects, 103--122, Proc. Sympos. Pure Math., \textbf{105}, Amer. Math. Soc., Providence, RI, (2023). 

\bibitem{CC5} A. Connes, C. Consani, \emph{Riemann-Roch for $\overline{\operatorname{Spec} \mathbf{Z}}$}, Bulletin des Sciences Mathematiques \textbf{187} (2023).

\bibitem{CC6} A. Connes, C. Consani, \emph{Geometry of the Scaling Site}, Selecta Math. (N.S.) \textbf{23} (2017), no. 3, 1803--1850.

\bibitem{CC7} A. Connes, C. Consani, \emph{Absolute Algebra and Segal's $\Gamma$-Rings: au Dessous de $\overline{Spec~\mathbf Z}$}, J. Number Theory {\bf 162} (2016), 518--551.

\bibitem{CC8} A. Connes,  C. Consani, \emph{Cyclic homology, Serre's local factors and  $\lambda$-operations}, 
J. K-Theory  \textbf{14} (2014), no. 1, 1--45.

\bibitem{CC9} A. Connes, C. Consani, \emph{On the Notion of Geometry over $\mathbf{F}_1$}, J. Algebraic Geometry \textbf{20} (2011), no. 3, 525-557.

\bibitem{CC10} A. Connes, C. Consani, \emph{Schemes over $\mathbf F_1$ and Zeta Functions}, Compositio Math. {\bf 146} (2010), no. 6, 1383--1415. 

\bibitem{H} L. Hesselholt, \emph{Topological Hochschild homology and the Hasse-Weil zeta function}, in An alpine bouquet of algebraic topology, 157–180, Contemp. Math., \textbf{708}, Amer. Math. Soc., (Providence), RI, (2018). 

\bibitem{S}
P. Scholze,  \emph{$p$-adic geometry}, Proceedings of the International Congress of Mathematicians—Rio de Janeiro 2018. Vol. I. Plenary lectures, 899–933, World Sci. Publ., Hackensack, NJ, 2018.

\bibitem{Si} C. Simpson, \emph{The Hodge filtration on nonabelian cohomology}, Algebraic Geometry- Santa Cruz 1995, Proc. Sympos. Pure Math., vol. {\bf 62}, Amer. Math. Soc., Providence, RI, 1997, pp. 217-281.
\end{thebibliography}
\end{document}